\documentclass[11pt, a4paper]{article}
\usepackage{amsthm}
\usepackage{amsmath,amssymb,latexsym,color}
\usepackage[mathscr]{eucal}
\usepackage[colorlinks,
            linkcolor=blue,
            anchorcolor=green,
            citecolor=magenta
           ]{hyperref}
\usepackage{graphicx}
\usepackage{tikz}
\usetikzlibrary{calc}
\usepackage{authblk}

\long\def\delete#1{}

\usepackage{color}

\definecolor{Blue}{rgb}{0,0,1}
\definecolor{Red}{rgb}{1,0,0}
\definecolor{DarkGreen}{rgb}{0,0.6,0}
\definecolor{DarkYellow}{rgb}{1,1,0.2}
\definecolor{DarkPurple}{rgb}{.6,0,1}

\usepackage{xcolor}
\usepackage[normalem]{ulem}

\usepackage{cleveref}
\crefformat{section}{\S#2#1#3}
\crefformat{subsection}{\S#2#1#3}
\crefformat{subsubsection}{\S#2#1#3}
\crefrangeformat{section}{\S\S#3#1#4 to~#5#2#6}
\crefmultiformat{section}{\S\S#2#1#3}{ and~#2#1#3}{, #2#1#3}{ and~#2#1#3}

\def\q{\hfill\rule{1ex}{1ex}}
\begin{document}
\setcounter{page}{1}
\newtheorem{thm}{Theorem}[section]
\newtheorem{fthm}[thm]{Fundamental Theorem}
\newtheorem{dfn}[thm]{Definition}
\newtheorem*{rem}{Remark}
\newtheorem{lem}[thm]{Lemma}
\newtheorem{cor}[thm]{Corollary}
\newtheorem{exa}[thm]{Example}
\newtheorem{prop}[thm]{Proposition}
\newtheorem{prob}[thm]{Problem}
\newtheorem{fact}[section]{Fact}
\newtheorem{con}[thm]{Conjecture}
\renewcommand{\thefootnote}{}
\newcommand{\remark}{\vspace{2ex}\noindent{\bf Remark.\quad}}
\newtheorem{ob}[thm]{Observation}
\newcommand{\rmnum}[1]{\romannumeral #1}
\renewcommand{\abovewithdelims}[2]{%
\genfrac{[}{]}{0pt}{}{#1}{#2}}

\newcommand\Sy{\mathrm{S}}
\newcommand\Cay{\mathrm{Cay}}
\newcommand\tw{\mathrm{tw}}
\newcommand\supp{\mathrm{supp}}

%-------------------  First Head  -----------------------------------------

\def\qed{\hfill$\Box$\vspace{11pt}}

\title {\bf  Treewidth of the generalized Kneser graphs}

\author{Ke Liu\thanks{E-mail: \texttt{liuke17@mails.tsinghua.edu.cn}}}
\author{Mengyu Cao\thanks{Corresponding author. E-mail: \texttt{caomengyu@mail.bnu.edu.cn}}}
\author{Mei Lu\thanks{E-mail: \texttt{lumei@tsinghua.edu.cn}}}

\affil{\small Department of Mathematical Sciences, Tsinghua University, Beijing 100084, China}

\date{}

\openup 0.5\jot
\maketitle

\begin{abstract}
Let $n$, $k$ and $t$ be integers with $1\leq t< k \leq n$. The \emph{generalized Kneser graph} $K(n,k,t)$ is a graph whose vertices are the $k$-subsets of a fixed $n$-set, where two $k$-subsets $A$ and $B$ are adjacent if $|A\cap B|<t$. The graph $K(n,k,1)$ is the well-known \emph{Kneser graph}. In 2014, Harvey and Wood determined the exact treewidth of the Kneser graphs for large $n$ with respect to $k$. In this paper, we give the exact treewidth of the generalized Kneser graphs for $t\geq2$ and large $n$ with respect to $k$ and $t$. In the special case when $t=k-1$, the graph $K(n,k,k-1)$ usually denoted by $\overline{J(n,k)}$ which is the complement of the Johnson graph $J(n,k)$. We give a more precise result for the exact value of the treewidth of $\overline{J(n,k)}$ for any $n$ and $k$.
%We also give the exact treewidth of the strong product of Kneser graphs $K(n,k,1)^{m\boxtimes}$. In the special case when $m=1$, our result gives rise to Harvey and Wood's result.

\vspace{2mm}


\


\end{abstract}

\section{Introduction}

Throughout this paper graphs are finite and undirected with no loops or multiple edges. The vertex and edge sets of a graph $G$ are denoted by $V(G)$ and $E(G)$, respectively. The numbers of vertices and edges of $G$ are denoted by $v(G)$ and $e(G)$, respectively. The degree of a vertex $x \in V(G)$ in $G$ is denoted by $d_{G}(x)$, and the edge joining vertices $u$ and $w$ are denoted as an unordered pair $\{u, w\}$. Let $\Delta(G)$ and $\delta(G)$ be the maximum and minimum degree of $G$, respectively. Especially, we call the vertices of the graph $T$ \emph{nodes} when $T$ is a tree.  Let $n$ and $k$ be integers with $1\leq k\leq n.$ Write $[n]=\{1,2,\ldots,n\}$ and denote by ${[n]\choose k}$ the family of all $k$-subsets of $[n].$ For any positive integer $t$, a family $\mathcal{F}\subseteq {[n]\choose k}$ is said to be \emph{$t$-intersecting} if $|A \cap B|\geq t$ for all $A, B\in\mathcal{F}.$ The \emph{complement} $\overline{G}$ of a graph $G$ has the same vertex set as $G$, where two vertices are adjacent in $\overline{G}$ if they are not adjacent in $G$. A set $S\subseteq V(G)$ is called \emph{independent set} if any pair of vertices in $S$ are non-adjacent in $G$.  The \emph{independence number} $\alpha(G)$ is the cardinality of maximum independent sets in $G$.

\begin{dfn}
{\em A \emph{tree decomposition} of a graph $G$ is a pair $(T,(B_{t})_{t\in V(T)})$, where $T$ is a tree and $(B_{t})_{t\in V(T)}$ is a family of subsets of $V(G)$ satisfying the following properties.
\begin{itemize}
\item[{\rm(i)}] For every $v\in V(G)$, the subgraph of $T$ induced by $B^{-1}(v)=\{t\in V(T)\mid v\in B_{t}\}$ is nonempty and connected.
\item[{\rm(ii)}] For every edge $\{u,w\}\in E(G)$, there is a $t\in V(T)$ such that $u,w\in B_{t}$.
\end{itemize}}
\end{dfn}

The \emph{width} of the decomposition $(T,(B_{t})_{t\in V(T)})$ is the number $\max\{|B_{t}|\mid t\in V(T)\}-1.$ The \emph{treewidth} of a graph $G$, denoted by $\tw(G)$, is the minimum width of the tree decompositions of $G$. By the definition, each graph $G$ has a tree decomposition $(T,(B_{t})_{t\in V(T)})$ where $T$ contains only one node $t$ with $B_{t}=V(G)$. Notice that this kind of decomposition has width $|V(G)|-1$ which is the maximum width of the graphs on $|V(G)|$ vertices and also the minimum width of the complete graph on $|V(G)|$ vertices. Therefore, the complete graph $K_n$ has treewidth $n-1$.

Treewidth is a well-studied parameter in modern graph theory that measures how ``tree-like" a graph is.  It is of fundamental importance in structural graph theory.
%which is initially defined by Halin \cite{Halin}. However, the definition gave by Halin is with different nomenclature to the modern standard. Later,
Robertson and Seymour used it in their famous series of papers proving the Graph Minor Theorem, for example, we refer the reader to  \cite{Robertson3,Robertson2,Robertson}.  Besides, Treewidth is also of key interest in the field of algorithm design. The problem of deciding whether a graph has tree decomposition of treewidth at most $k$ is NP-complete \cite{A} and the treewidth is regarded as a key parameter in fixed-parameter tractability. There are many NP-complete problems are solvable in polynomial time on graphs of bounded treewidth \cite{Bodlaender,Bo}. In the past few decades, there are lots of literatures investigate the treewidth of certain graphs, for example, \cite{Wood,Wood2,Kloks,Li,Mitsche,Wood3}. However, it is difficult to determine the treewidth exactly in most situations, and there are only few papers obtained the exact value of the treewidth of some certain graphs. In 2014, Harvey and Wood determined the exact treewidth of the Kneser graphs which is ${n\choose k}-{n-1\choose k-1}-1$ for $n\geq 4k^2-4k+3$ \cite{Wood}. Motivated by this result, we study the exact value of treewidth of the generalized Kneser graphs for $t\geq2$ in this paper.

Let $n$, $k$ and $t$ be integers with $1\leq t< k \leq n$. The \emph{generalized Kneser graph} $K(n,k,t)$ is a graph whose vertices are the $k$-subsets of a fixed $n$-set, where two $k$-subsets $A$ and $B$ are adjacent if $|A\cap B|<t$. The graph $K(n,k,1)$ is the well-known \emph{Kneser graph}. Kneser  graphs  were  first  investigated  by  Kneser  \cite{Kneser}.  There is a famous result of the  chromatic  number  of  $K(n,k,1)$ which  was  shown  to  be $n-2k+2$  by  Lov\'{a}sz  \cite{Lovasz},  as  Kneser  originally conjectured. The generalized Kneser graphs as the main generalization of the Kneser graphs are also widely studied. The  famous  Erd\H{o}s-Ko-Rado  Theorem \cite{Erdos-Ko-Rado-1961-313} has a well-known relationship to the independent number of the generalized Kneser graphs, since an independent set in the generalized Kneser graph $K(n,k,t)$ is a $t$-intersecting family of ${[n]\choose k}$. In the special case when $t=k-1$, the graph $K(n,k,k-1)$, usually denoted by $\overline{J(n,k)}$ is the complement of the Johnson graph $J(n,k)$. The \emph{Johnson graph} $J(n,k)$ is the graph whose vertices are the $k$-subsets of a fixed $n$-set as well, where two vertices $A$ and $B$ are adjacent if $|A\cap B|=k-1$. Over the years several aspects of Johnson graphs such as chromatic number, connectivity, eigenvalues, automorphisms, regular embeddings and some other properties have been widely studied as one can find in, for example, \cite{Bitan, Brouwer2,Dabrowski,Daven,Jones,Numata}. These graphs are important because they enable us to translate many combinatorial problems about finite sets into graph theory, such as the context of coding theory and design theory. Also, we see
that many interesting objects from finite geometry occur encoded as cliques and independent sets in these graphs, and this leads to interesting variants of the Erd\H{o}s-Ko-Rado Theorem.

The main results in this paper are as follows.

\begin{thm}\label{GKneser}
Let $n$, $k$ and $t$ be positive integers with $n\geq 2(k-t)(t+1){k\choose t}+k+t+1$. Let $K(n,k,t)$ be the generalized Kneser graph for $k>t\geq 2$. Then
$$\tw(K(n,k,t))={n\choose k}-{n-t\choose k-t}-1.$$
\end{thm}

%\begin{thm}\label{Sproduct}
%Let $n$, $k$ and $m$ be positive integers with $n\geq 4m(k-1)k^m+k+1$. Then
%$$\tw(K(n,k)^{m\boxtimes})={n\choose k}^m-{n-1\choose k-1}^m-1.$$
%\end{thm}
%In the special case when $t=1$, Theorem~\ref{Sproduct} gives rise to Theorem 1 in \cite{Wood}.

From  the proof of Theorem~\ref{GKneser}, one can obtain that $\tw(K(n,k,1))={n\choose k}-{n-1\choose k-1}-1$ for $n\geq 4k^2-3k+2$, which equals to the value of treewidth of Kneser graphs given by Harvey and Wood \cite{Wood}. However, the lower bound of $n$ in our result  is slightly bigger than that of the result of Harvey and Wood (they let $n\geq 4k^2-4k+3$). Therefore, we only consider the situation $t\geq 2$.

In the special case when $t=k-1$, the graph $K(n,k,k-1)$ is the complement of the Johnson graph $J(n,k)$. We give a more precise result for the exact value of the treewidth of $\overline{J(n,k)}$ for any $n$ and $k$. Note that $\overline{J(n,k)}$ is an empty graph when $n<k+2$. Thus we only consider the case with $n\geq k+2$.

\begin{thm}\label{CJohnson}
Let $n$ and $k$ be positive integers with $n\geq k+2$ and $k\geq 2$. Let $\overline{J(n,k)}$ be the complement of Johnson graph $J(n,k)$. Then
$$
{\tw}(\overline{J(n,k)}) = \begin{cases}
%{n\choose 3}-n+1, & \mbox{if}\ k=3 \mbox{and}\ n\geq 7,\\
%{n\choose 2}-n, & \mbox{if}\ k=2  \mbox{and}\ n\geq 6,\\
1, & \mbox{if}\ k=2 \ \mbox{and}\ n=4,\\
14, & \mbox{if}\ k=3 \ \mbox{and}\ n= 6,\\
4, & \mbox{if}\ k=3 \ \mbox{and}\ n=5, \ \mbox{or}\ k=2 \ \mbox{and}\ n=5,\\
{n\choose k}-n+k-2, & \mbox{if}\ k\geq 4 \ \mbox{and}\ n\geq 2k, \ \mbox{or}\ k\in\{2,3\} \ \mbox{and}\ n\geq k+4,\\
{n\choose k}-k-2, & \mbox{if}\ k\geq 4 \ \mbox{and}\ n< 2k.
\end{cases}
$$
\end{thm}

%\begin{cor}\label{ekr}
%\end{cor}

The rest of this paper is organized as follows. In the next section we will prove some important inequalities which are useful for the proof in Subsection \ref{LB1}. In Section \ref{thm1}, we will give the exact treewidth of the generalized Kneser graphs for $k>t\geq2$ and large $n$ corresponding to $k$ and $t$. After that, we will study the treewidth of the complement of Johnson graphs for any possible $n$ and $k$ in Section \ref{thm2}.

\section{Preliminaries}
In this section,  we will prove a number of inequalities. In Subsection~\ref{LB1}, we will use these inequalities to prove a lower bound for the treewidth of the generalized Kneser graphs in Theorem~\ref{GKneser}.

\begin{lem}\label{lem1}
Let $n$, $k$ and $t$ be positive integers. If $n\geq 2(k-t)(t+1){k\choose t}+k+t+1$, then $$\frac{1}{2{k\choose t}}{n-t\choose k-t}>{n-t\choose k-t}-{n-2t-1\choose k-t}.$$
\end{lem}
\begin{proof}
Firstly, we prove a claim.

\noindent\textbf{Claim 1.} For any $i\in\{0,1,\hdots,k-t-1\}$, we have $$\frac{n-t-i}{n-2t-1-i}<\sqrt[k-t]{1+\frac{1}{2{k\choose t}-1}}.$$

{\bf Proof of Claim 1} It is clear that $\frac{n-t-i}{n-2t-1-i}$ is increasing as $i\in\{0,1,\hdots,k-t-1\}$ increases. Therefore, it suffices to prove that $$\frac{n-k+1}{n-k-t}<\sqrt[k-t]{1+\frac{1}{2{k\choose t}-1}}.$$

Let $c=\sqrt[k-t]{1+\frac{1}{2{k\choose t}-1}}.$ We have
\begin{align*}
\frac{t+1}{c-1}=&\frac{(t+1)(c^{k-t-1}+c^{k-t-2}+\cdots+c+1)}{c^{k-t}-1}\\
<&\frac{(t+1)(k-t)c^{k-t}}{c^{k-t}-1}\\
=&2(t+1)(k-t){k\choose t}.
\end{align*}

If $n\geq 2(k-t)(t+1){k\choose t}+k+t+1$, then $n-k-t>2(k-t)(t+1){k\choose t}>\frac{t+1}{c-1}$, which implies that
\begin{gather*}
(c-1)n-(c-1)k-(c-1)t>t+1,\\
cn-ck-ct>n-k+1.
\end{gather*}
This yields that $\frac{n-k+1}{n-k-t}<c$. We complete the proof of the claim.\q

By  Claim 1, we have
$$\frac{(n-t)(n-t-1)\cdots(n-t+1)}{(n-2t-1)(n-2t-2)\cdots(n-k-t)}<1+\frac{1}{2{k\choose t}-1}$$ holds. Thus we get $$\frac{{n-t\choose k-t}}{{n-2t-1\choose k-t}}<1+\frac{1}{2{k\choose t}-1},$$ which implies that $$\left(\frac{1}{2{k\choose t}}-1\right){n-t\choose k-t}>-{n-2t-1\choose k-t},$$
as required.
\end{proof}

\begin{lem}\label{lem2}
Let $n$, $k$ and $t$ be positive integers. Let $p$ be a fixed constant with $\frac{2}{3}\leq p < 1$. If $n\geq \frac{1}{1-p}(k-t)(k+1)+2t$, then $$(1-p){n-t\choose k-t}\geq(k+1)\sum_{a=0}^{t-2}{n-2t-1\choose {n-k-t+a}}+t(k+1){n-2t-1\choose n-k-1}.$$
\end{lem}
\begin{proof}
Since ${n-t\choose k-t}=\sum_{i=0}^{t}{t \choose i}{n-2t \choose k-t-i}$, it is suffices to prove that
$$(1-p)\sum_{i=0}^{t}{t \choose i}{n-2t \choose k-t-i}\geq (k+1)\sum_{a=0}^{t-2}{n-2t-1\choose k-t-1-a}+t(k+1){n-2t-1\choose k-2t}.$$

Since ${n-2t\choose k-t-i}=\frac{n-2t}{k-t-i}{n-2t-1\choose k-t-i-1}$ and $n\geq \frac{1}{1-p}(k-t)(k+1)+2t,$ we have $$(1-p){t \choose i}{n-2t \choose k-t-i}\geq (k+1){n-2t-1\choose k-t-1-i}$$  for $i\in \{0,1,\hdots, t-2\}$. Furthermore, since $n\geq \frac{1}{1-p}(k-t)(k+1)+2t \geq \frac{1}{1-p}(k-2t+1)(k+1)+2t$, we have $$(1-p){t \choose i}{n-2t \choose k-t-i}\geq t(k+1){n-2t-1\choose k-2t}$$ for $i=t-1$.

Therefore, the required result holds.
\end{proof}

%\begin{lem}\label{lem2}
%Let $n$, $k$ and $t$ be positive integers. Let $p$ be a fixed constant with $\frac{2}{3}\leq p < 1$. If $n\geq \frac{1}{1-p}t(k-t)(k+1)+t+1$, then $$(1-p){n-t\choose k-t}>(k+1)t{n-2t-1\choose k-t-1}.$$
%\end{lem}
%\begin{proof}
%Since ${n-t\choose k-t}=\frac{n-t}{k-t}{n-t-1\choose k-t-1},$ it is suffices to prove that
%$$\frac{(1-p)(n-t)}{t(k-t)(k+1)}>\frac{{n-2t-1\choose k-t-1}}{{n-t-1\choose k-t-1}}.$$
%
%As $n\geq \frac{1}{1-p}t(k-t)(k+1)+t+1$, we have
%$$\frac{(1-p)(n-t)}{t(k-t)(k+1)}>1>\frac{{n-2t-1\choose k-t-1}}{{n-t-1\choose k-t-1}}$$
%holds.
%\end{proof}

\section{Treewidth of the generalized Kneser graphs}\label{thm1}

\subsection{upper bound for treewidth in Theorem~\ref{GKneser}}
In this subsection, we will give an upper bound for the treewidth of the graph $K(n,k,t)$ with the help of the famous Erd\H{o}s-Ko-Rado Theorem for finite sets.

\begin{thm}{\rm(Erd\H{o}s-Ko-Rado Theorem~\cite{Erdos-Ko-Rado-1961-313,Wilson-1984})}\label{EKR}
Let $n, k$ and $t$ be positive integers with $1 \leq t < k \le n$. If $n\geq (t +1)(k-t +1)$ and $\mathcal{F}\subseteq{[n]\choose k}$ is a $t$-intersecting family, then
$$
|\mathcal{F}|\leq{n-t\choose k-t}.
$$
Moreover, if $n> (t +1)(k-t +1)$, equality holds if and only if $\mathcal{F}$ consists of all $k$-subsets that contain
a fixed $t$-subset of $[n]$.
\end{thm}
The result of Erd\H{o}s-Ko-Rado Theorem for finite sets is clearly equivalent to the independent number of the generalized Kneser graph $K(n,k,t)$. That is,
\begin{equation}\label{independent}
\alpha(K(n,k,t))={n-t\choose k-t}
\end{equation}
for $n\geq (t +1)(k-t +1)$.

\begin{prop}{\rm (\cite{Wood})}\label{upperbound}
For any graph $G$, $\tw(G)\leq\max\{\Delta(G),|V(G)|-\alpha(G)-1\}$.
\end{prop}

\begin{lem}\label{upbound}
If $n$, $k$ and $t$ are integers with $k>t \geq1$ and $n\geq 2(k-t)(t+1){k\choose t}+k+t+1$, then
\begin{equation}\label{upperbound1}
\tw(K(n,k,t))\leq {n\choose k}-{n-t\choose k-t}-1.
\end{equation}
\end{lem}
\begin{proof}
According to Proposition~\ref{upperbound}, to prove an upper bound of $\tw(K(n,k,t))$ we only need to compare the size of $\Delta(K(n,k,t))$ and $|V(K(n,k,t))|-\alpha(K(n,k,t))-1.$

For any $k$-subset $A\in {[n]\choose k}$, define $N(A)$ to be the set of all the neighbors of $A$ in $K(n,k,t)$. Thus,
\begin{align*}
N(A)=\left\{F\in {[n]\choose k}\mid |F\cap A|<t\right\}
=\bigcup_{i=0}^{t-1}\left\{F\in {[n]\choose k}\mid |F\cap A|=i\right\}.
\end{align*}
Since $|\{F\in {[n]\choose k}\mid |F\cap A|=i\}|={k\choose i}{n-k\choose k-i}$, and $K(n,k,t)$ is a regular graph, we have
$$\Delta(K(n,k,t))=|N(A)|=\sum_{i=0}^{t-1}{k\choose i}{n-k\choose k-i}.$$

\noindent\textbf{Claim 2.}  $\frac{(n-k-i)(k-i)}{(n-t-i)(k-t-i)}>1$ for any $i\in\{0,1,\hdots,k-t-1\}$.

\vskip.2cm
{\bf Proof of Claim 2} Firstly, we have
\begin{align*}
(n-k-i)(k-i)-(n-t-i)(k-t-i)=tn-t^2+tk-k^2+(k-2t)i
\end{align*}
for any $i\in\{0,1,\hdots,k-t-1\}$.  We divide the proof of this claim into the following two cases.

\noindent\textbf{Case 1.} $k\geq 2t.$

In this case, since $n\geq 2(k-t)(t+1){k\choose t}+k+t+1$ and $k\geq t+1$, we have
\begin{align*}
&(n-k-i)(k-i)-(n-t-i)(k-t-i)\\
\geq& tn-t^2+tk-k^2\\
\geq& t(2(k-t)(t+1)k+t)-t^2-k(k-t)\\
>&0.
\end{align*}

\noindent\textbf{Case 2.} $k<2t.$

In this case, similarly, since $n\geq 2(k-t)(t+1){k\choose t}+k+t+1$, we have
\begin{align*}
&(n-k-i)(k-i)-(n-t-i)(k-t-i)\\
\geq& tn-t^2+tk-k^2+(k-2t)(k-t-1)\\
=& tn+t^2+2t-2kt-k\\
>&0.
\end{align*}

Therefore, we obtain $(n-k-i)(k-i)-(n-t-i)(k-t-i)>0$ for any $i\in\{0,1,\hdots,k-t-1\}$, and the result of this claim follows.\q

By  Claim 2, we have $$\frac{(n-k)(n-k-1)\cdots(n-2k+t+1)k(k-1)\cdots(t+1)}{(n-t)(n-t-1)\cdots(n-k+1)(k-t)(k-t-1)\cdots1}>1$$ holds, which implies that  ${k\choose t}{n-k\choose k-t}-{n-t\choose k-t}\geq0.$

Furthermore, since ${n\choose k}=\sum_{i=0}^{k}{k\choose i}{n-k\choose k-i}$, by (\ref{independent}) and $k\geq t+1$, we have
\begin{align*}
&|V(K(n,k,t))|-\Delta(K(n,k,t))-\alpha(K(n,k,t))-1\\
=&\sum_{i=t}^{k}{k\choose i}{n-k\choose k-i}-{n-t\choose k-t}-1\\
\geq&{k\choose t}{n-k\choose k-t}+{k\choose t+1}{n-k\choose k-t-1}-{n-t\choose k-t}-1\\
\geq& 0.
\end{align*}

Therefore, $|V(K(n,k,t))|-\alpha(K(n,k,t))-1\geq\Delta(K(n,k,t)),$ yielding that $\tw(K(n,k,t))\leq |V(K(n,k,t))|-\alpha(K(n,k,t))-1={n\choose k}-{n-t\choose k-t}-1,$ as required.
\end{proof}
\subsection{lower bound for treewidth in Theorem~\ref{GKneser}}\label{LB1}
In this subsection, we will study the lower bound for treewidth of the generalized Kneser graphs.

Let $\mathcal {F}$ be a family of $k$-subsets of $[n]$, the \emph{$t$-shadow} of $\mathcal {F}$ is defined as
$$\partial_t(\mathcal {F})=\left\{X\in{[n]\choose t}\mid X\subseteq F \mbox{ for some } F\in\mathcal {F}\right\}.$$
According to the definition, the $t$-shadow $\partial_t(\mathcal {F})$ of $\mathcal {F}$ contains all $t$-subsets that are contained in one of the $k$-subsets of $\mathcal {F}$. The \emph{complement} of $F\in{[n]\choose k}$ is the $(n-k)$-set $\overline{F}:=[n]\setminus F$. Define the \emph{complement} of $\mathcal {F}\subseteq {[n]\choose k}$ by $\overline{\mathcal {F}}:=\{\overline{F}\in{[n]\choose n-k}\mid \overline{F}\ \mbox{is  the complement of some } F\in\mathcal {F}\}$.  Given a $t$-subset $X\in{[n]\choose t}$, denote $\mathcal {F}_X$ the set of all $k$-subsets of $\mathcal {F}$ containing $X$, and let $\mathcal {F}_{-X}:=\mathcal {F}\setminus \mathcal {F}_X.$ Assume that $\mathcal {F}_{-X}\neq \emptyset$, let $\overline{\mathcal {F}_{-X}}:=\{\overline{F}\in{[n]\choose n-k}\mid \overline{F}=[n]\setminus F, \mbox{ where}\ F\in\mathcal {F}_{-X}\}$ and  $\overline{\mathcal {F}_{-X}}^*:=\{\overline{F}^*\in\bigcup_{a\in\{0,1,\hdots,t-1\}}{[n]\choose n-k-t+a}\mid \overline{F}^*=\overline{F}\setminus X,\mbox{ where }\ \overline{F}\in\overline{\mathcal {F}_{-X}}\mbox{ such that }\ |X\cap \overline{F}|=t-a\}$. Similarly, let $\mathcal {F}_{X}^*:=\{F^*\in{[n]\choose k-t} \mid F^*=F\setminus X,\mbox{ where }\ F\in\mathcal {F}_{X}\}$. Then $|\mathcal {F}_{X}^*|=|\mathcal {F}_{X}|$.

Define the \emph{colexicographic ordering}, colex ordering for short, on the $k$-subsets of $[n]$ as follows: if $F_1,F_2\in{[n]\choose k}$ are distinct, then $F_1 < F_2$ when $\max\{F_1\setminus F_2\}<\max\{F_2\setminus F_1\}$.  Thus this is a strict total order. Let $\mathcal {F}\subseteq {[n]\choose k}$, $\mathcal {F}$ is \emph{first} if $\mathcal {F}$ consists of the first $|\mathcal {F}|$ $k$-subsets of $[n]$ in the colex ordering.
%If we let $\max\{F_1\setminus F_2\}=0$ when $F_1\setminus F_2=\emptyset$, notice that this is also a strict total order on $\mathcal {F}\subseteq\bigcup\limits_a{[n]\choose a}$, where $a\in\{1,2,\hdots,n\}.$ And the definition of ``first" is similar.

\begin{prop}{\rm (\cite{Katona,Kruskal})}\label{shadow}
Let $\mathcal {F}\subseteq {[n]\choose k}$, and $\partial_t(\mathcal {F})$ be the $t$-shadow of $\mathcal {F}$. If $|\mathcal {F}|$ is a fixed constant, then $|\partial_t(\mathcal {F})|$ is minimised when $\mathcal {F}$ is first.
\end{prop}

For $\mathcal {F}\subseteq\bigcup\limits_i{[n]\choose k+i}$, in the following lemma we can prove a more general result than Proposition \ref{shadow}, where $i\in\{0,1,\hdots,t-1\}.$

\begin{lem}\label{shadow2}
Let $k,g$ be positive integers with $g< k$. Let $\mathcal {A}\subseteq \bigcup\limits_i{[n]\choose k+i}$ and $\mathcal {A}=\mathcal {A}_0\cup\mathcal {A}_1\cup\hdots\cup\mathcal {A}_{t-1}$, where $\mathcal {A}_i\subseteq {[n]\choose k+i}$ for any $i\in\{0,1,\hdots,t-1\}$. Let $\mathcal {S}_i=\partial_{g}(\mathcal {A}_i)$ be the $g$-shadow of $\mathcal {A}_i$ and $\mathcal {S}=\partial_{g}(\mathcal {A})$ be the $g$-shadow of $\mathcal {A}$. If $|\mathcal {A}_i|$ is a fixed constant, then $|\mathcal {S}|$ is minimised when $\mathcal {A}_i$ is first in the colex ordering for any $i\in\{0,1,\hdots,t-1\}$.
\end{lem}
\begin{proof}
Notice that $\mathcal {S}=\partial_{g}(\mathcal {A})=\mathcal {S}_0\cup\mathcal {S}_1\cup\hdots\cup\mathcal {S}_{t-1}$.  According to Proposition~\ref{shadow}, we have $|\mathcal {S}_i|$ is minimised when $\mathcal {A}_i$ is first in the colex ordering.

We assume that $\mathcal {A}_i$ is first for any $i\in\{0,1,\hdots,t-1\}$, we prove that there exists $i\in\{0,1,\hdots,t-1\}$ such that $\mathcal {S}_i=\bigcup_{j\in \{0,1,\hdots,t-1\}\setminus i} \mathcal {S}_j$. Let $\alpha$ be the maximum element in $\mathcal {A}_0\cup\mathcal {A}_1\cup\hdots\cup\mathcal {A}_{n-k}$ in the colex ordering. Let $\mathcal {A}_i$ be the one containing $\alpha$ and let $\max\alpha$ be the maximum element in $\alpha$. Therefore, it is sufficient to prove that $\mathcal {S}_i\supseteq \mathcal {S}_j$ for any $j\in\{0,1,\hdots,t-1\}$ and $i\neq j$. Notice that $\max\alpha \geq k+i$ and $\max\beta \leq\max \alpha$ for any $\beta\in \mathcal {S}_j$ by the choice of $\alpha$.

\noindent\textbf{Case 1.} $\max\alpha = k+i.$

We have $\alpha=\{1,2\hdots,k+i\}$ since $\mathcal {A}_i$ is first. Therefore, in this case for any $\beta \in \mathcal {S}_j$ we have $\beta\subseteq \alpha$; otherwise there exists $l\in\beta\setminus \alpha$ such that $l>\max\alpha$, a contradiction. This implies that $\beta\in \mathcal {S}_i.$

\noindent\textbf{Case 2.} $\max\alpha > k+i$.

Let $a=\max\{\alpha\setminus\beta\}$. Notice that there exists a $(k+i)$-set $\gamma\in{[\max\alpha]\setminus \{a\}\choose k+i}$ containing $\beta$. Since $\mathcal {A}_i$ is first, we have $\gamma<\alpha$ in the colex ordering which implies that $\gamma\in \mathcal {A}_i$. Thus, we also have $\beta \in \mathcal {S}_i$ in this case.

Therefore, $\mathcal {S}_i\supseteq \mathcal {S}_j$, as required. Thus we have $|\mathcal {S}|=|\mathcal {S}_i|$ when $\mathcal {A}_i$ is first in the colex ordering for any $i\in\{0,1,\hdots,t-1\}$, where $\mathcal {S}_i$ is the one contains $\bigcup_{j\in \{0,1,\hdots,t-1\}\setminus i} \mathcal {S}_j$.
\end{proof}

Let $X$ be a subset of $V(G)$. The graph $G-X$ is a subgraph of $G$ induced by $V(G)\setminus X$. Let $G[X]$ be the subgraph of $G$ induced by $X$. Let $p$ be a fixed constant with $\frac{2}{3}\leq p < 1$. The \emph{$p$-separator} of  $G$ is  a subset $X\subset V(G)$ such that there is no component in $G-X$ that contains more than $p|V(G-X)|$ vertices. If $|X|\leq c$, we call $X$ a $p$-separator of order $c$. There is a well-known relationship between the treewidth and the $p$-separators of  $G$.

\begin{prop}{\rm (\cite{Robertson})}\label{separator}
Every graph $G$ has a $p$-separator of order $\tw(G)+1$ for each $\frac{2}{3}\leq p < 1$.
\end{prop}

With the help of these important results we give the following lemma.

\begin{lem}\label{lowerbound}
Let $n$, $k$, $t$ and $p$ be integers with $k>t \geq1$, $\frac{2}{3}\leq p < 1$ and $n\geq \max\{2(k-t)(t+1){k\choose t}+k+t+1,\frac{1}{1-p}(k-t)(k+1)+2t\}$. If $X$ is a $p$-separator of $K(n,k,t)$, then $$|X|\geq {n\choose k}-{n-t\choose k-t}.$$
\end{lem}
\begin{proof}
Suppose to the contrary that $|X|< {n\choose k}-{n-t\choose k-t}$. Thus we have
\begin{align}\label{assum}
|V(K(n,k,t)-X)|>{n-t\choose k-t}.
\end{align}
Since $X$ is a $p$-separator of $K(n,k,t)$, we can partition the components of $K(n,k,t)-X$ into two parts such that the components in each part contain at most $p|V(K(n,k,t)-X)|$ vertices. Therefore, $V(K(n,k,t)-X)$ can be partitioned into two parts $\mathcal {A}$ and $\mathcal {B}$ such that there is no edge between $\mathcal {A}$ and $\mathcal {B}$. This implies that $|u\cap v|\geq t$ for any $u\in \mathcal {A}$ and $v\in \mathcal {B}$. Furthermore, we have
\begin{align}
(1-p)|V(K(n,k,t)-X)|\leq &|\mathcal {A}| \leq \frac{1}{2}|V(K(n,k,t)-X)|, \label{e2}\\
\frac{1}{2}|V(K(n,k,t)-X)|\leq &|\mathcal {B}| \leq p|V(K(n,k,t)-X)|. \label{e3}
\end{align}

Since $\mathcal {A}$  and $\mathcal {B}$ are both non-empty by (\ref{e2}) and (\ref{e3}), respectively, there is a vertex $u\in \mathcal {A}$ such that $|u\cap v|\geq t$ for any vertex $v\in \mathcal {B}$. Therefore, according to the Pigeonhole Principle, there exists a $t$-subset $Y$ such that
it belongs to at least $\frac{1}{{k\choose t}}|\mathcal {B}|$ vertices in $\mathcal {B}$. Thus, we have $|\mathcal {B}_Y|\geq \frac{1}{{k\choose t}}|\mathcal {B}|.$ According to $(\ref{assum})$ and $(\ref{e3})$, we have $|\mathcal {B}|\geq \frac{1}{2}|V(K(n,k,t)-X)|> \frac{1}{2}{n-t\choose k-t}$, which implies that
\begin{equation}\label{lb1}
|\mathcal {B}_Y^*|=|\mathcal {B}_Y|\geq\frac{1}{{k\choose t}}|\mathcal {B}|>\frac{1}{2{k\choose t}}{n-t\choose k-t}.
\end{equation}

Notice that $\overline{\mathcal {A}_{-Y}}^*$ is a family of $(n-k-t+a)$-subsets of $[n]$, where $a\in\{0,1,\hdots,t-1\}$. Let $\overline{\mathcal {A}_{-Y}}^*=\mathcal {A}_0\cup\mathcal {A}_1\cup\hdots\cup\mathcal {A}_{t-1}$, where $\mathcal {A}_a\subseteq {[n]\choose n-k-t+a}$ and $a\in\{0,1,\hdots,t-1\}$. We have the following claim.

%Next, we consider the upper bound of $|\mathcal {B}_Y|.$
\vskip.2cm
\noindent\textbf{Claim 3.} $|\overline{\mathcal {A}_{-Y}}^*|=\sum_{a=0}^{t-1}|\mathcal {A}_a|<\sum_{a=0}^{t-1}{n-2t-1\choose n-k-t+a}.$ Furthermore, we have $|\mathcal {A}_a|<{n-2t-1\choose n-k-t+a}$ for every $a\in\{0,1,\hdots,t-1\}$.
%$|\mathcal {A}_{-Y}|< t{n-2t-1\choose k-t-1}$.
\vskip.2cm

{\bf Proof of Claim 3} Clearly, we just need to consider the case $\mathcal {A}_{-Y}\neq \emptyset$. Then $\overline{\mathcal {A}_{-Y}}$ is non-empty since $|\mathcal {A}_{-Y}|=|\overline{\mathcal {A}_{-Y}}|$. Let $w\in \mathcal {B}_Y$ and $\overline{z}\in \overline{\mathcal {A}_{-Y}}$ satisfying $|Y\cap \overline{z}|=t-a$, where $a\in\{0,1,\hdots,t-1\}$ and $z\in \mathcal {A}_{-Y}$. Then $wz\notin E(K(n,k,t))$. Let $w*=w\setminus Y$ and $\overline{z}^*=\overline{z}\setminus Y$. If $w*\subseteq \overline{z}^*$, then
\begin{align*}
|w\cap \overline{z}|&=|(w^*\cup Y)\cap (\overline{z}^*\cup (Y \cap \overline{z}))|\\
&=|w^*\cup Y|+|\overline{z}^*\cup (Y \cap \overline{z})|-|w^*\cup Y\cup \overline{z}^*\cup (Y \cap \overline{z})|\\
&=|w|+|\overline{z}^*|+|Y \cap \overline{z}|-|\overline{z}^*\cap (Y \cap \overline{z})|-|Y\cup \overline{z}^*|\\
&=k+(n-k-t+a)+(t-a)-(n-k+a)\\
&=k-a,
\end{align*}
which implies that $|w\cap z|=k-(k-a)=a<t$ and that is  $wz\in E(K(n,k,t))$, a contradiction. Thus, we have $w*\nsubseteq \overline{z}^*$ for any $w*\in \mathcal {B}_Y^*$ and $\overline{z}^*\in \overline{\mathcal {A}_{-Y}}^*$. Therefore, if we let $\mathcal {S}:=\partial_{k-t}(\overline{\mathcal {A}_{-Y}}^*)$ be the $(k-t)$-shadow of $\overline{\mathcal {A}_{-Y}}^*$, then we have $w*\nsubseteq s$ for any $w*\in \mathcal {B}_Y^*$ and $s\in \mathcal {S}$, which implies that
\begin{equation}\label{e4}
|\mathcal {B}_Y^*|\leq {n-t\choose k-t}-|\mathcal {S}|.
\end{equation}
Hence can obtain an upper bound of $|\mathcal {B}_Y^*|$ by taking $|\mathcal {S}|$ to be minimised.
%According to Proposition~\ref{shadow}, we have $|\mathcal {S}|$ is minimised when $\overline{\mathcal {A}_{-Y}}^*$ is first in the colex ordering.
%Using this result we give the following upper bound for $|\mathcal {A}_{-Y}|$ and give a lower bound for $|\mathcal {S}|$.

%We claim that $|\mathcal {A}_{-Y}|< t{n-2t-1\choose k-t-1}.$

Suppose to the contrary that $$|\overline{\mathcal {A}_{-Y}}^*|\geq\sum_{a=0}^{t-1}{n-2t-1\choose n-k-t+a}.$$

First we prove  that $|\mathcal {S}|\geq{n-2t-1\choose k-t}$. Let $\mathcal {S}=\partial_{k-t}(\overline{\mathcal {A}_{-Y}}^*)=\mathcal {S}_0\cup\mathcal {S}_1\cup\hdots\cup\mathcal {S}_{t-1}$, where $\mathcal {S}_a$ is the $(k-t)$-shadow of $\mathcal {A}_a$.  According to Lemma~\ref{shadow2}, we have $|\mathcal {S}|$ is minimised when $\mathcal {A}_a$ is first in the colex ordering for any $a\in\{0,1,\hdots,t-1\}$.  Therefore, to obtain a lower bound of $|\mathcal {S}|$, we assume that $\mathcal {A}_a$ is first for any $a\in\{0,1,\hdots,t-1\}$. As $|\overline{\mathcal {A}_{-Y}}^*|\geq \sum_{a=0}^{t-1}{n-2t-1\choose n-k-t+a},$ there exists at least one $\mathcal {A}_a$ such that $|\mathcal {A}_a|\geq {n-2t-1\choose n-k-t+a}$, implying that $\mathcal {A}_a$ contains the first ${n-2t-1\choose n-k-t+a}$ $(n-k-t+a)$-sets in the colex ordering. This implies that $\mathcal {A}_a$ contains all $(n-k-t+a)$-subsets of $[n-2t-1]$. Thus $\mathcal {S}_a$ contains all $(k-t)$-subsets of $[n-2t-1]$, which follows that $|\mathcal {S}_a|\geq{n-2t-1\choose k-t}$. Therefore, we have $|\mathcal {S}|\geq{n-2t-1\choose k-t}$, as required.

Next, by the lower bound of $|\mathcal {S}|$ and $(\ref{e4})$, we have the following upper bound for $|\mathcal {B}_Y^*|$.
\begin{equation}\label{ub1}
|\mathcal {B}_Y^*| \leq {n-t\choose k-t}-{n-2t-1\choose k-t}.
\end{equation}
However, combining with $(\ref{lb1})$ and $(\ref{ub1})$, by $n\geq2(k-t)(t+1){k\choose t}+k+t+1$ and Lemma~\ref{lem1}, we have a contradiction. Thus, we have the claim holds. \q
\iffalse
\begin{equation}\label{A-Y}
|\mathcal {A}_{-Y}|< t{n-2t-1\choose k-t-1},
\end{equation}
as required.
\fi

%Thus $|\mathcal {A}_{Y}|\geq k|\mathcal {A}_{-Y}|$, which implies that $(k+1)|\mathcal {A}_{Y}|\geq k|\mathcal {A}|$ by $|\mathcal {A}|=|\mathcal {A}_{Y}|+|\mathcal {A}_{-Y}|$, a contradiction. Therefore the result holds.

\vskip.2cm
\noindent\textbf{Claim 4.} $|\mathcal{A}_{-Y}|<\sum_{a=0}^{t-2}{n-2t-1\choose n-k-t+a}+t{n-2t-1\choose n-k-1}$.
\vskip.2cm

{\bf Proof of Claim 4}
If $|\mathcal {A}_{-Y}|=|\overline{\mathcal {A}_{-Y}}^*|$, we have the conclusion. If $|\mathcal {A}_{-Y}|\neq |\overline{\mathcal {A}_{-Y}}^*|$, there must be $v_1,v_2\in \mathcal {A}_{-Y}$ such that $\overline{v_1}^*=\overline{v_2}^*$, which implies $v_1\setminus Y=v_2\setminus Y$. Let $|v_1\cap Y|=|v_2\cap Y|=r$, where $r\in\{1,2,\ldots,t-1\}$. We say $r= t-1$. Suppose to the contrary that $r\leq t-2$.

For any $u\in B_Y$, since $u$ is not adjacent to $v_1,v_2$, we have $|u\cap v_1|\geq t$ and $|u\cap v_2|\geq t$, which implies $|u\cap (v_1\setminus Y)|\geq t-r$. Therefore, according to the Pigeonhole Principle, there exists a $(t-r)$-subset $Z$ such that it belongs to at least $\frac{1}{{k-r\choose t-r}}|\mathcal{B}_Y|$ vertices in $\mathcal{B}_{Y}$. Thus we have
\begin{equation}\label{lbb1}
|\mathcal{B}_{Y\cup Z}|\geq \frac{1}{{k-r\choose t-r}}|B_Y|>\frac{1}{2{k-r\choose t-r}{k\choose t}}{n-t\choose k-t},
\end{equation}
by $(\ref{lb1})$.

On the other hand, we know $\mathcal{B}_{Y\cup Z}$ contains the fixed $2t-r$ elements of $Y\cup v_1$, thus
\begin{equation}\label{lbb2}
|\mathcal{B}_{Y\cup Z}|\leq {n-2t+r\choose k-2t+r}.
\end{equation}

Comparing with the lower bound and the upper bound of $\mathcal{B}_{Y\cup Z}$ in $(\ref{lbb1})$ and $(\ref{lbb2})$, we have
$$\frac{1}{2{k\choose t}}{n-t\choose k-t}
\leq {k-r\choose t-r}{n-2t+r\choose k-2t+r}.$$
Since ${k-r\choose t-r}{n-2t+r\choose k-2t+r}$ is increasing as $a\in\{1,2,\hdots,t-2\}$ increases, we get
$$\frac{1}{2{k\choose t}}{n-t\choose k-t}
\leq {k-t+2\choose 2}{n-t-2\choose k-t-2},$$
a contradiction with $n\geq2(k-t)(t+1){k\choose t}+k+t+1$ and we have $r=t-1$. Recall that $\overline{\mathcal {A}_{-Y}}^*=\mathcal {A}_0\cup\mathcal {A}_1\cup\hdots\cup\mathcal {A}_{t-1}$, where $\mathcal {A}_a\subseteq {[n]\choose n-k-t+a}$ and $a\in\{0,1,\hdots,t-1\}$. Then we know for every $v\in \mathcal{A}_b$, there exists exactly one $w\in \mathcal{A}_{-Y}$ such that $\overline{w}^*=v$, where $b\in\{0,1,\hdots,t-2\}$. And for each $v\in \mathcal{A}_{t-1}$, there exists at most ${t\choose t-1}$ $w\in \mathcal{A}_{-Y}$ such that $\overline{w}^*=v$, which implies $|\mathcal{A}_{-Y}|\leq \sum_{a=0}^{t-2}|\mathcal{A}_a|+t|\mathcal{A}_{t-1}|$. By Claim $3$, we know $|\mathcal{A}_{t-1}|<{n-2t-1 \choose n-k-1}$ and $|\mathcal{A}_{-Y}|<\sum_{a=0}^{t-2}{n-2t-1\choose n-k-t+a}+t{n-2t-1\choose n-k-1}$.\q

Next, we prove that
\begin{equation}\label{e5}
|\mathcal {A}_{Y}|\geq\frac{k}{k+1}|\mathcal {A}|.
\end{equation}
using Claim $4$. Assume for the sake of contradiction that $|\mathcal {A}_{Y}|<\frac{k}{k+1}|\mathcal {A}|$. Then we have $|\mathcal {A}|< (k+1)|\mathcal {A}_{-Y}|$ since $|\mathcal {A}|=|\mathcal {A}_{Y}|+|\mathcal {A}_{-Y}|$. By  Claim $4$, we have $|\mathcal {A}|<(k+1)\sum_{a=0}^{t-2}{n-2t-1\choose n-k-t+a}+(k+1)t{n-2t-1\choose n-k-1}.$ On the other hand, by $(\ref{e2})$, we have $|\mathcal {A}|\geq (1-p)|K(n,k,t)-X|$. Therefore, it follows that
$$(1-p){n-t\choose k-t}<(k+1)\sum_{a=0}^{t-2}{n-2t-1\choose n-k-t+a}+(k+1)t{n-2t-1\choose n-k-1}.$$
However, by $n\geq \frac{1}{1-p}(k-t)(k+1)+2t$ and  Lemma \ref{lem2}, we get a contradiction.

\vskip.2cm
\noindent\textbf{Claim 5.} $\mathcal {B}_{Y}=\mathcal {B}$.
\vskip.2cm

{\bf Proof of Claim 5} We suppose to the contrary that  $\mathcal {B}_{Y}\neq\mathcal {B}$. Therefore, there exists some $v\in \mathcal {B}$ such that $Y\nsubseteq v.$ For any $u\in \mathcal {A}_{Y}$, since $u$ is not adjacent to $v$ in $K(n,k,t)-X$, we have $|u\cap v|\geq t$. Let $|Y\cap v|=a$, $a\in\{0,1,\hdots,t-1\}$. Thus $u$ contains at least $t-a$ elements of $v\setminus Y$ as $|u\cap v|\geq t$. It follows that
\begin{equation}\label{Ayu}
|\mathcal {A}_{Y}|\leq {k-a\choose t-a}{n-t-(t-a)\choose k-t-(t-a)}.
\end{equation}

On the other hand, combining $(\ref{e5})$ with $(\ref{assum})$ and $(\ref{e2})$,  we get
\begin{equation}\label{Ayl}
|\mathcal {A}_{Y}|\geq \frac{k}{k+1}|\mathcal {A}|> \frac{(1-p)k}{k+1}{n-t\choose k-t}.
\end{equation}

Comparing with the upper bound and the lower bound of $|\mathcal {A}_{Y}|$ in $(\ref{Ayu})$ and $(\ref{Ayl})$, we have
$$\frac{(1-p)k}{k+1}{n-t\choose k-t}<{k-a\choose t-a}{n-2t+a\choose k-2t+a}.$$
Since  ${k-a\choose t-a}{n-2t+a\choose k-2t+a}$ is increasing as $a\in\{0,1,\hdots,t-1\}$ increases, we get
$$\frac{(1-p)k}{k+1}{n-t\choose k-t}< (k-t+1){n-t-1\choose k-t-1}.$$
By ${n\choose k}=\frac{n}{k}{n-1\choose k-1}$, we have $n< \frac{k-t+1}{(1-p)k}(k-t)(k+1)+t< \frac{1}{1-p}(k-t)(k+1)+t$, a contradiction with $n\geq \frac{1}{1-p}(k-t)(k+1)+2t$. Thus, we obtain $\mathcal {B}_{Y}=\mathcal {B}$, as required. \q

\vskip.2cm
\noindent\textbf{Claim 6.} $\mathcal {A}_{Y}=\mathcal {A}$.
\vskip.2cm

{\bf Proof of Claim 6} This claim follows by essentially the similar argument as Claim $5$ above. If $\mathcal {A}_{-Y}=\emptyset,$ then we have the claim holds. If $\mathcal {A}_{-Y}\neq\emptyset,$ we suppose to the contrary that $\mathcal {A}_{Y}\neq\mathcal {A}$. Therefore, there exists some $w\in \mathcal {A}$ such that $Y\nsubseteq w$, and for any $z\in \mathcal {B}_{Y}$, $|w\cap z|\geq t$. By Claim $5$, $(\ref{assum})$ and $(\ref{e3})$, we have $|\mathcal {B}_{Y}|=|\mathcal {B}|> \frac{1}{2}{n-t\choose k-t}$. On the other hand, let $|Y\cap w|=a$, $a\in\{0,1,\hdots,t-1\}$. Thus $z$ contains at least $t-a$ elements of $w\setminus Y$ since $|w\cap z|\geq t$. It follows that $|\mathcal {B}_{Y}|\leq {k-a\choose t-a}{n-t-(t-a)\choose k-t-(t-a)},$ this equals to the upper bound of $|\mathcal {A}_{Y}|$ in the previous proof of Claim $5$. Therefore, combining with the upper bound and the lower bound of $|\mathcal {B}_{Y}|$, we have
$$\frac{1}{2}{n-t\choose k-t}< {k-a\choose t-a}{n-2t+a\choose k-2t+a}.$$
Since ${k-a\choose t-a}{n-2t+a\choose k-2t+a}$ is increasing as $a\in\{0,1,\hdots,t-1\}$ increases, we obtain $n< 2k^2-4kt+2t^2+2k-t$, a contradiction. Therefore, we get $\mathcal {A}_{Y}=\mathcal {A}$, as required. \q

From Claims $5$ and $6$, we have every vertex in $\mathcal {A}\cup \mathcal {B}=\mathcal {A}_{Y}\cup \mathcal {B}_{Y}$ contains $Y$. Since $V(K(n,k,t)-X)$ can be partitioned into $\mathcal {A}$ and $\mathcal {B}$, we get $|K(n,k,t)-X|=|\mathcal {A}_{Y}|+|\mathcal {B}_{Y}|\leq{n-t\choose k-t}$, which implies that $|X|\geq {n\choose k}-{n-t\choose k-t}$, a contradiction.
\end{proof}

\noindent\textit{Proof of Theorem~\ref{GKneser}.}\quad
%Since it is clear to show that
%$$2(k-t)(t+1){k\choose t}+k+t+1\geq \frac{1}{1-p}(k-t)(k+1)+t+1$$
%for every $\frac{2}{3}\leq p <1$ and $t\geq 2$,
By Lemma \ref{lowerbound}, if we let $X$ be a $\frac{2}{3}$-separator of $K(n,k,t)$, then $|X|\geq {n\choose k}-{n-t\choose k-t}.$  Since $2(k-t)(t+1){k\choose t}+k+t+1\geq \frac{1}{1-p}t(k-t)(k+1)+t+1$ for $k>t\geq2$, we have $\tw(K(n,k,t))\geq {n\choose k}-{n-t\choose k-t}-1$ for $n\geq 2(k-t)(t+1){k\choose t}+k+t+1$ by Proposition \ref{separator}. Next, combining with the upper bound of $\tw(K(n,k,t))$ in Lemma~\ref{upbound}, we obtain the result directly. \qed

In the special case when $t=1$, by Lemma \ref{lowerbound}, if we let $X$ be a $\frac{2}{3}$-separator of $K(n,k,1)$, then $|X|\geq {n\choose k}-{n-1\choose k-1}$ for $n\geq4k^2-3k+2.$ Therefore, we have $\tw(K(n,k,1))={n\choose k}-{n-1\choose k-1}-1$ for $n\geq4k^2-3k+2$ by Proposition \ref{separator} and Lemma~\ref{upbound}.

\section{Treewidth of the complement of Johnson graphs}\label{thm2}
In this section, we study the treewidth of the complement of Johnson graphs, and give the exact value of the treewidth of $\overline{J(n,k)}$ for $n\geq k+2$. Note that $\overline{J(n,k)}$ is an empty graph when $n<k+2$. Firstly, we can easily get the upper bound of $\tw(\overline{J(n,k)})$ as follow according to Theorem~\ref{EKR} and Proposition~\ref{upperbound}.

\begin{lem}\label{up2}
Let $n$ and $k$ be positive integers with $k\geq2$ and $n\geq 2k$. Then
$$\tw(\overline{J(n,k)})\leq{n\choose k}-n+k-2.$$
\end{lem}

\begin{lem}\label{lb2}
Let $n$ and $k$ be positive integers with $k\geq2$ and $n\geq \max\{k+4,2k\}$. Then
$$\tw(\overline{J(n,k)})\geq{n\choose k}-n+k-2.$$
\end{lem}
\begin{proof}
We suppose to the contrary that $\tw(\overline{J(n,k)})<{n\choose k}-n+k-2.$ By Proposition \ref{separator}, there exists a $\frac{2}{3}$-separator $X$ such that $|X|<{n\choose k}-n+k-1$. Therefore, $|V(\overline{J(n,k)}-X)|>n-k+1$. Furthermore, since $n\geq \max\{k+4,2k\}\geq k+4$, we have $$|V(\overline{J(n,k)}-X)|\geq 6.$$

By the similar analysis of the proof of Lemma~\ref{lowerbound}, it is easy to see that $V(\overline{J(n,k)}-X)$ can be partitioned into two parts $\mathcal {A}$ and $\mathcal {B}$ such that there is no edge between $\mathcal {A}$ and $\mathcal {B}$, and the equations
\begin{align}
\frac{1}{3}|V(\overline{J(n,k)}-X)|\leq &|\mathcal {A}|,|\mathcal {B}|\leq \frac{2}{3}|V(\overline{J(n,k)}-X)| \label{e2'}
\end{align}
holds. Thus, $|\mathcal {A}|,|\mathcal {B}|\geq 2$. By Theorem~\ref{EKR} and $n\geq \max\{k+4,2k\}\geq 2k$, we have  $\alpha(\overline{J(n,k)})=n-k+1$. Thus $V(\overline{J(n,k)}-X)$ is too large to be an independent set which implies that there exists an edge in the subgraph induced by $\mathcal {A}$ or the subgraph induced by $\mathcal {B}$. Without loss of generality, assume that this edge is in $\mathcal {A}$ and let the two endpoints of the edge are
\begin{align*}
v_1=\{1,2,\hdots,x,a_1,\hdots, a_{k-x}\} \ \mbox{and}\ v_2=\{1,2,\hdots,x,b_1,\hdots, b_{k-x}\},
\end{align*}
where $v_1\cap v_2=\{1,2,\hdots,x\}$ and $a_i\neq b_j$ for any $1\leq i,j \leq k-x$ and $x\leq k-2.$

We claim that $x=k-2.$ Since for any vertex $w\in \mathcal {B}$ there is no vertex in $\mathcal {A}$ is adjacent to $w$ in $\overline{J(n,k)}-X$, we have $|w\cap v_1|,|w\cap v_2|\geq k-1$. This implies that $|w\cap v_1|,|w\cap v_2|= k-1$ and then $|v_1\cap v_2|\geq k-2$, as $w$ is a $k$-set. On the other hand, $x\leq k-2$ from above. Thus we have $x=k-2,$ as required.

Next, we prove that $v_1\cap v_2=\{1,2,\hdots,k-2\}\subseteq w$ for any vertex $w\in \mathcal {B}$. As $w\cap v_1\cap v_2\subseteq \{1,2,\hdots,k-2\}$, we only need to prove $|w\cap v_1\cap v_2|\geq k-2.$ We have $$|w\cap v_1\cap v_2|=|w\cap v_1|+|w\cap v_2|-|w\cap(v_1\cup v_2)|\geq 2(k-1)-k=k-2,$$ as required.

Therefore, we obtain
\begin{align*}
\mathcal {A}\supseteq \big\{&v_1=\{1,2,\hdots,k-2,a_1,a_2\},v_2=\{1,2,\hdots,k-2,b_1,b_2\}\big\},\\
\mathcal {B}\subseteq \big\{&\{1,2,\hdots,k-2,a_1,b_1\},\{1,2,\hdots,k-2,a_1,b_2\},\\
&\{1,2,\hdots,k-2,a_2,b_1\},\{1,2,\hdots,k-2,a_2,b_2\}\big\}.
\end{align*}

We first have  the following two claims. Let $G=\overline{J(n,k)}-X$ for short.

\vskip.2cm
\noindent\textbf{Claim 7.} If  $G[\mathcal {B}]$ contains an edge, then $n=k+4$, $|V(G)|=n-k+2$  and one of  $G[\mathcal {A}]$ and $G[\mathcal {B}]$ is connected.
\vskip.2cm

{\bf Proof of Claim 7}
According to the above analysis, in this situation $V(G)\subseteq \{\{1,2,\hdots,k-2,i,j\}\mid i,j\in\{a_1,a_2,b_1,b_2\}\}.$ Thus, $|V(G)|\leq 6$. Since $|V(G)|\geq n-k+2\geq 6$, we have $|V(G)|= 6$ and then $n=k+4$. Therefore, there exists three edges in $G$ in total. If there is one edge in  $G[\mathcal {A}]$ (that is $(v_1,v_2)$), and the other two edges in  $G[\mathcal {B}]$, then  $G[\mathcal {A}]$ is connected. Otherwise,  $G[\mathcal {B}]$ contains exactly one edge and it is connected.\q

\vskip.2cm
\noindent\textbf{Claim 8.} If $\mathcal {B}$ is an independent set in $G$, then $|V(G)|=n-k+2$, and  $G[\mathcal {A}]$ is connected.
\vskip.2cm

{\bf Proof of Claim 8} Without loss of generality, let $$\mathcal {B}=\big\{\{1,2,\hdots,k-2,a_1,b_1\},\{1,2,\hdots,k-2,a_1,b_2\}\big\}.$$ Thus, we have $$\mathcal {A}\subseteq \big\{\{1,2,\hdots,k-2,a_1,i\}\mid i\in [n]\setminus \{1,2,\hdots,k-2,a_1,b_1,b_2\}\big\}\cup \big\{\{1,2,\hdots,k-2,b_1,b_2\}\big\}.$$
Therefore, $|V(G)|=|\mathcal {A}|+|\mathcal {B}|\leq (n-k-1)+1+2=n-k+2$. On the other hand, since $|V(G)|\geq n-k+2$, we have $|V(G)|=n-k+2$ and $\mathcal {A}$ is exactly that set above. Thus we have the claim holds.\q
%If $n>k+4$, then $|V(\overline{J(n,k)}-X)|>6$, which implies that $|\mathcal {B}|\geq \frac{7}{3}$ by (\ref{e2'}), a contradiction with $|\mathcal {B}|=2$.} Consequently, we also have the subgraph of $\overline{J(n,k)}-X$ induced by $\mathcal {A}$ is connected and $|V(\overline{J(n,k)}-X)|=n-k+2$ in this case.

Let $(T,(B_{t})_{t\in V(T)})$ be a minimum width tree decomposition for $\overline{J(n,k)}$, such that if $t_1t_2\in V(T)$, then $B_{t_1}\nsubseteq B_{t_2}$. According to the assumption that $\tw(\overline{J(n,k)})<{n\choose k}-n+k-2$, we have $|B_t|\leq {n\choose k}-n+k-2$ for all $t\in V(T)$. Since there is a fact that $X\subset V(\overline{J(n,k)})$ is a subset of some bag $B_t$ and  $|V(G)|=n-k+2$, by Claims 7 and 8, it follows that $|X|={n\choose k}-n+k-2$ and $X$ is a bag with maximum order, that is $X=B_t$.  By Claims 7 and 8, we have
$G[\mathcal {A}]$ (resp.  $G[\mathcal {B}]$) is a  component of $G$
if  $G[\mathcal {A}]$ is connected (resp. if  $G[\mathcal {B}]$ is connected). Then there is a subtree in $T-t$ contains all vertices of $\mathcal {A}$ (resp. $\mathcal {B}$).  Notice that every vertex in $X$ has a neighbor in $\mathcal {A}$ (resp. $\mathcal {B}$) if $G[\mathcal {A}]$ is connected (resp. if  $G[\mathcal {B}]$ is connected)  since the vertices which are non-adjacent with the vertices in $\mathcal {A}$ are all in $\mathcal {B}$. Let $t'$ be the node of this subtree adjacent to $t$. Thus we have $B_t\subseteq B_{t'}$, a contradiction.
\end{proof}

\begin{figure}
\begin{center}
\begin{tikzpicture}[sibling distance=19em,
  every node/.style = {shape=rectangle, rounded corners,
    draw, align=center,
    top color=white, bottom color=gray!16}]]
  \node {\footnotesize{$\big\{\{1,2,5\}\big\}\cup(X\setminus \big\{\{2,4,5\}\big\})$}}
    child { node {\large{$X$}}
      child { node {\footnotesize{$\big\{\{1,2,3\}\big\}\cup(X\setminus \big\{\{1,3,4\}\big\})$}}
           child { node {\footnotesize{$\big\{\{1,2,3\},\{1,4,5\}\big\}\cup(X\setminus \big\{\{1,3,4\},\{2,4,5\}\big\})$}}}}
        child { node {\footnotesize{$\big\{\{1,2,4\}\big\}\cup(X\setminus \big\{\{1,3,4\}\big\})$}}
           child { node {\footnotesize{$\big\{\{1,2,4\},\{1,3,5\}\big\}\cup(X\setminus \big\{\{1,3,4\},\{2,3,5\}\big\})$}}}}
        };
\end{tikzpicture}
\end{center}
\caption{Tree decomposition}
\label{fig1}
\end{figure}
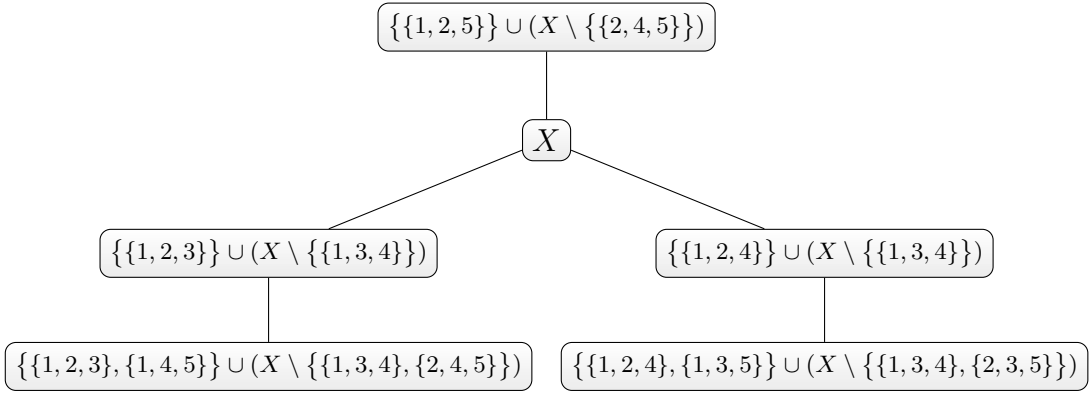

\noindent\textit{Proof of Theorem~\ref{CJohnson}.}\quad We divide the proof of this theorem into the following two cases.

\vskip.2cm
\noindent\textbf{Case 1.} $k\geq 4.$
\vskip.2cm

In this case we have $2k\geq k+4$. If $n\geq 2k$, by Lemmas \ref{up2} and \ref{lb2}, we have $\tw(\overline{J(n,k)})={n\choose k}-n+k-2.$ Note that $\overline{J(n,k)}\cong\overline{J(n,n-k)}$. If $n<2k$ then  $n>2(n-k)$.  By  Lemmas \ref{up2} and \ref{lb2},  $\tw(\overline{J(n,k)})=\tw(\overline{J(n,n-k)})={n\choose k}-k-2.$
%According to the fact, if we consider the treewidth for $\overline{J(n,k)}$ when $n<2k$ then that is equivalent to consider the treewidth for $\overline{J(n,n-k)}$ when $n>2(n-k)$. Thus we have $\tw(\overline{J(n,k)})={n\choose k}-k-2.$

\vskip.2cm
\noindent\textbf{Case 2.} $1<k<4.$
\vskip.2cm

We have $k+4> 2k$ in this case.

\vskip.2cm
\noindent\textbf{Subcase 1.} $k=2.$

If $n\geq 6=k+4$, by Lemmas \ref{up2} and \ref{lb2}, we get $\tw(\overline{J(n,2)})={n\choose k}-n+k-2={n\choose 2}-n.$ If $n=4$ or $5$, then $\tw(\overline{J(4,2)})=1$ and $\tw(\overline{J(5,2)})=4$ \cite{Wood}, respectively.

\vskip.2cm
\noindent\textbf{Subcase 2.} $k=3.$

If $n\geq 7=k+4$, by Lemmas \ref{up2} and \ref{lb2} again, we obtain $\tw(\overline{J(n,2)})={n\choose 3}-n+1.$

If $n=5$, by $\overline{J(5,3)}\simeq \overline{J(5,2)}$,  we have  $\tw(\overline{J(5,3)})=\tw(\overline{J(5,2)})=4$ \cite{Wood}. Notice that $\overline{J(5,3)}$ is the Peterson graph and it is easy to construct a minimal width tree decomposition for it shown in Figure~\ref{fig1}, where $$X=\big\{\{1,3,4\},\{2,3,4\},\{2,3,5\},\{2,4,5\},\{3,4,5\}\big\}.$$

If $n=6$, we can easily construct a tree decomposition for $\overline{J(6,3)}$, and that also satisfies the structure in Figure~\ref{fig1}, where
$$X={[6]\choose 3}\setminus\big\{\{1,2,3\},\{1,2,4\},\{1,2,5\},\{1,3,5\},\{1,4,5\}\big\}.$$
This implies that $\tw(\overline{J(6,3)})\leq 14$.

Suppose to the contrary that $\tw(\overline{J(6,3)})<14.$ By Proposition \ref{separator}, there exists a $\frac{2}{3}$-separator $X$ such that $|X|<15$. Therefore, $|V(\overline{J(6,3)}-X)|\geq 6$ and $V(\overline{J(6,3)}-X)$ can be partitioned into two parts $\mathcal {A}$ and $\mathcal {B}$ such that there is no edge between $\mathcal {A}$ and $\mathcal {B}$, and the equations $(\ref{e2'})$ holds. Thus, $|\mathcal {A}|,|\mathcal {B}|\geq 2$. Denote $G=\overline{J(6,3)}-X$ for short.
By Theorem~\ref{EKR},  $\alpha(\overline{J(6,3)})=4$. Thus  there exists an edge $(v_1,v_2)$ in  $G[\mathcal {A}]$ or  $G[\mathcal {B}]$. Without loss of generality, assume that  $(v_1,v_2)$ is in $G[\mathcal {A}]$. Since for any vertex $w\in \mathcal {B}$ there is no vertex in $\mathcal {A}$ is adjacent to $w$ in $G$, we have $|w\cap v_1|,|w\cap v_2|\geq 2$. This implies that $|w\cap v_1|=|w\cap v_2|= 2$ and then $|v_1\cap v_2|\geq 1$, as $w$ is a $3$-set. Note that $|v_1\cap v_2|\leq 1$. Therefore, we have $|v_1\cap v_2|=1$. Without loss of generality, let
\begin{align*}
v_1=\{1,a_1, a_2\} \ \mbox{and}\ v_2=\{1,b_1, b_2\},
\end{align*}
where $a_i\neq b_j$ for any $1\leq i,j \leq 2$.

Since for any vertex $w\in \mathcal {B}$, $|w\cap v_1\cap v_2|=|w\cap v_1|+|w\cap v_2|-|w\cap(v_1\cup v_2)|\geq 1$  and  $w\cap v_1\cap v_2\subseteq \{1\}$, we have $v_1\cap v_2=\{1\}\subseteq w$. Therefore, we obtain
\begin{align*}
\mathcal {A}\supseteq \big\{&v_1=\{1,a_1,a_2\},v_2=\{1,b_1,b_2\}\big\},\\
\mathcal {B}\subseteq \big\{&\{1,a_1,b_1\},\{1,a_1,b_2\},\{1,a_2,b_1\},\{1,a_2,b_2\}\big\}.
\end{align*}

If  $G[\mathcal {B}]$ contains an edge, then $V(G)\subseteq \{\{1,i,j\}\mid i,j\in\{a_1,a_2,b_1,b_2\}\}$. Thus, $|V(G)|\leq 6$. Since $|V(G)|\geq 6$, we have $|V(G)|= 6$. Therefore, there exists three edges in $G$ in total. Without loss of generality, assume that there is one edge in  $G[\mathcal {A}]$ (that is $(v_1,v_2)$), and the other two edges are in  $G[\mathcal {B}]$. Then $G[\mathcal {A}]$ is connected. If $\mathcal {B}$ is an independent set in $G$, without loss of generality, let $$\mathcal {B}=\big\{\{1,a_1,b_1\},\{1,a_1,b_2\}\big\}.$$ Thus, we have $$\mathcal {A}\subseteq \big\{\{1,a_1,i\}\mid i\in [n]\setminus \{1,a_1,b_1,b_2\}\big\}\cup \big\{\{1,b_1,b_2\}\big\}.$$
Therefore, $|V(G)|=|\mathcal {A}|+|\mathcal {B}|\leq 5$, which contradicts with  $|V(G)|\geq 6$. Consequently, we have $|V(G)|=6$  and  $G[\mathcal {A}]$ is connected.

Let $(T,(B_{t})_{t\in V(T)})$ be a minimum width tree decomposition for $\overline{J(6,3)}$, such that if $t_1t_2\in V(T)$, then $B_{t_1}\nsubseteq B_{t_2}$. Since $\tw(\overline{J(6,3)})<14$, we have $|B_t|\leq 14$ for all $t\in V(T)$. Since there is a fact that $X\subset V(\overline{J(6,3)})$ is a subset of some bag $B_t$ and  $|V(G)|=6$, it follows that $|X|=14$ and $X$ is a bag with maximum order, that is $X=B_t$. Since $G[\mathcal {A}]$ is a  component of $G$, there is a subtree of $T-t$ contains all vertices of $\mathcal {A}$.  Notice that every vertex in $X$ has a neighbor in $\mathcal {A}$ since the vertices which are non-adjacent with the vertices in $\mathcal {A}$ are all in $\mathcal {B}$. Let $t'$ be the node of this subtree adjacent to $t$. Thus we have $B_t\subseteq B_{t'}$, a contradiction.

Consequently, we complete the proof of this theorem. \qed

\iffalse
Notice that follow similar analysis of the proof of Theorem \ref{CJohnson}, we can also give the complete result for the exact value of the treewidth of $K(n,k,k-2)$ for any $n$ and $k$. Note that $K(n,k,k-2)$ is an empty graph when $n<k+3$. Thus we only consider the case with $n\geq k+3$.

\begin{rem}
Let $n$ and $k$ be positive integers with $n\geq k+2$ and $k\geq 2$. Let $\overline{J(n,k)}$ be the complement of Johnson graph $J(n,k)$. Then
$$
{\tw}(\overline{J(n,k)}) = \begin{cases}
%{n\choose 3}-n+1, & \mbox{if}\ k=3 \mbox{and}\ n\geq 7,\\
%{n\choose 2}-n, & \mbox{if}\ k=2  \mbox{and}\ n\geq 6,\\
1, & \mbox{if}\ k=2 \ \mbox{and}\ n=4,\\
14, & \mbox{if}\ k=3 \ \mbox{and}\ n= 6,\\
4, & \mbox{if}\ k=3 \ \mbox{and}\ n=5, \ \mbox{or}\ k=2 \ \mbox{and}\ n=5,\\
{n\choose k}-n+k-2, & \mbox{if}\ k\geq 4 \ \mbox{and}\ n\geq 2k, \ \mbox{or}\ k\in\{2,3\} \ \mbox{and}\ n\geq k+4,\\
{n\choose k}-k-2, & \mbox{if}\ k\geq 4 \ \mbox{and}\ n< 2k.
\end{cases}
$$
\end{rem}
\fi
\section*{Acknowledgement}
This research was supported by   the National Natural Science Foundation of China (Grant 11771247 \& 11971158) and  Tsinghua University Initiative Scientific Research Program.

\end{document}